\documentclass{amsart}
\usepackage{graphicx}
\usepackage{amssymb, amsmath, amsthm,amscd,epsfig,graphicx}
\usepackage{verbatim}
\usepackage{palatino}
\usepackage{parskip}
\usepackage[english]{babel}
\usepackage{dcpic,pictexwd}
\begin{document}

\title{From $(\mathbb{Z},X)$-modules to homotopy cosheaves}
\thanks{The author was supported by Michael Weiss' Alexander von Humboldt Professorship.}

\author{Filipp Levikov
}

\address{
              Freie Universit\"{a}t Berlin, Mathematisches Institut\\
              Arnimallee 7, 14195 Berlin}
              \email{filipp.levikov@fu-berlin.de}           

\keywords{$L$-theory, homotopy cosheaves, total surgery obstruction}\subjclass[2010]{57R67, 19J25, 19G24,  55U15}
\begin{abstract}
We construct a functor from the category of $(\mathbb{Z},X)$-modules of Ranicki (cf. \cite{Ra92}) to the category of homotopy cosheaves of chain complexes of Ranicki-Weiss (cf. \cite{RaWei10}) inducing an equivalence on $L$-theory. The $L$-theory of $(\mathbb{Z},X)$-modules is central in the algebraic formulation of the surgery exact sequence and in the construction of the total surgery obstruction by Ranicki, as described in \cite{Ra79}. The symmetric $L$-theory of homotopy cosheaf complexes is used by Ranicki-Weiss in \cite{RaWei10}, to reprove the topological invariance of rational Pontryagin classes. The work presented here may be considered as an addendum to the latter article and suggests some translation of ideas of Ranicki into the language of homotopy chain complexes of cosheaves.

\end{abstract}

\maketitle
\newcommand{\SymL}{\mathbf{L}^{\bullet}}
\newcommand{\QuadL}{\mathbf{L}_{\bullet}}
\newcommand{\LO}[1]{[{#1}]_{\mathbf{L}}}
\newcommand{\cat}[1]{\mathbb{#1}}
\newcommand{\cattwo}[1]{\mathcal{#1}}
\newcommand{\md}[1]{\mid{#1}\mid}
\newcommand{\Open}{\mathcal{O}}
\newcommand{\rOpen}{\mathcal{O}^{\textit{restr}}}
\newcommand{\ChAb}{\mathop{Ch}(\mathcal{A}b)}
\newcommand{\Cone}{\mathop{Cone}}
\newcommand{\restr}[1]{{#1}^\textit{restr}}
\newcommand{\opst}[1]{\mathring{\mathop{st}}({#1})}
\newcommand{\coend}[1]{\int^{#1}}
\newcommand{\hocoend}[1]{\mathrm{ho}\!\!\int^{#1}}

\newcommand{\SC}[2]{C[{#1}]\otimes S_*({#1}\cap{#2})}
\newcommand{\SD}[2]{D[{#1}]\otimes S_*({#1}\cap{#2})}
\newcommand{\SCxrel}[4]{({#1}[{#2}]/\!\!\!\sum_{\rho\not\in {#3}}{#1}({#2}))\otimes S_*({#2}\cap{#4},{#2}\cap{#4}\setminus{#3})}

\newcommand{\hocolim}[1][]{\mathop{\mathrm{hocolim}}_{#1}}
\newcommand{\holim}[1][]{\mathop{\mathrm{holim}}\limits_{#1}}
\newcommand{\holimtwo}[2]{\mathop{\mathrm{holim}}_{\substack{{#1}\\{#2}}}}
\newcommand{\limtwo}[2]{\mathop{\mathrm{lim}}_{\substack{{#1}\\{#2}}}}
\newcommand{\colimtwo}[2]{\mathop{\mathrm{colim}}_{\substack{{#1}\\{#2}}}}
\newcommand{\hocolimtwo}[2]{\mathop{\mathrm{hocolim}}_{\substack{{#1}\\{#2}}}}
\newcommand{\sumtwo}[2]{{\mathop{\sum}}_{\substack{{#1}\\{#2}}}}

\newcommand{\undercat}[2]{{#1}\!\downarrow\!{#2}}
\newcommand{\Hom}{\mathrm{Hom}}
\newcommand{\coeq}{\mathop{\mathrm{coeq}}}
\newcommand{\eq}{\mathop{\mathrm{eq}}}
\newcommand{\Tot}{\mathcal{T}\!\!ot}
\newcommand{\srep}{\mathcal{S}\!rep}
\newcommand{\csrep}{c\mathcal{S}\!rep}

\newcommand{\boxtimesRA}{\mathop{\boxtimes_{\mathcal{K}_{(\mathbb{Z},X)}}}}
\newcommand{\ZZtwo}{\mathbb{Z}[\mathbb{Z}_2]}
\newcommand{\twoline}[2]
{\left\{\begin{tabular}{l}
  {#1} \\
  {#2} \\
   \end{tabular}
\right.
}
\newcommand{\id}{\mathop{\mathrm{id}}}
\newcommand{\catZXmod}{\mathcal{K}_{(\mathbb{Z},X)}}
\theoremstyle{definition}
\newtheorem{definition}{Definition}[section]
\newtheorem{example}[definition]{Example}
\newtheorem{remark}[definition]{Remark}
\newtheorem{corollary}[definition]{Corollary}
\newtheorem{theorem}[definition]{Theorem}
\theoremstyle{plain}
\newtheorem{lemma}[definition]{Lemma}
\newtheorem{oplemma}[definition]{! Lemma}
\newtheorem{proposition}[definition]{Proposition}


 \newcommand{\colim}[1][]{\mathop{\mathrm{colim}}\limits_{#1}}
 \newcommand{\real}[1]{\vert{#1}\vert}

\section{Introduction}
\label{intro}
In \cite{RaWei10}, Ranicki and Weiss reprove the topological invariance of rational Pontrjagin classes by constructing for a topological manifold $M$ a symmetric $L$-theory orientation whose rationalization is identified with the Poincar\'e dual of the total $L$-class. For a locally compact, Hausdorff and separable
space $X$, they introduce the category $\mathcal{D}_X$ of
``cosheaf''-like complexes of abelian groups and a framework for defining $L$-theory in this setting. 
All this structure is assembled into what in the following is called the weak algebraic bordism category $\cattwo{K}_X$. 
Associating to $X$ the symmetric $L$-theory spectrum of $\cattwo{K}_X$ 
gives a functor from spaces to spectra\linebreak
\[
X \mapsto \SymL(\mathcal{K}_X)
\]
which is homotopy invariant and excisive and thus is equivalent to symmetric $L$-homology. Although not dealt with in \cite{RaWei10} the corresponding functor to quadratic $L$-theory 
\[
X \mapsto \QuadL(\mathcal{K}_X)
\] 
is constructed in an analogous way. 
On the
other hand, 
for a realisation of a simplicial complex $X$ the framework of \cite{Ra92} leads to
the definition of a symmetric (resp. quadratic) $L$-theory spectrum of the algebraic bordism category $\cattwo{K}_{(\mathbb{Z},X)}$ \footnote{In the original source \cite{Ra92} the category is denoted by $\Lambda_\ast(\mathbb{Z},X)$ with the cagegory of $(\mathbb{Z},X)$-modules $(\cat{A}(\mathbb{Z},X))$ being the underlying additive category with chain duality. See section \ref{Algebraic Bordism Categories Basics} for the notation.} of chain
complexes of $\mathbb{Z}$-modules over $X$. For a simplicial complex there
are functors to spectra
\[
X \mapsto \SymL(\cattwo{K}_{(\mathbb{Z},X)}) \qquad \textnormal{ and } \qquad X \mapsto \QuadL(\cattwo{K}_{(\mathbb{Z},X)})
\]
where the quadratic $L$-homology description is even more important since it is part of the construction of Ranicki's total surgery obstruction $s(X)$. Although it is never put like this in \cite{RaWei10} the authors set out to achieve the goal described in \cite[p.19]{Ra92}: ``the ultimate version of the algebraic $L$-theory assembly map should be topologically invariant, using the language of sheaf theory [\ldots]''. Since the symmetric (resp. quadratic) $L$-groups of the above categories (i.e. the homotopy groups of the corresponding $L$-theory spectra) are just the
$\SymL$- (resp. $\QuadL$-\nolinebreak) homology groups, they coincide abstractly. Let $\cat{B}(\mathbb{Z},X)$ denote the category of chain complexes underlying $\cattwo{K}_{(\mathbb{Z},X)}$. The goal of this article is to construct an explicit, geometric natural transformation of functors
\[
\cat{B}(\mathbb{Z}, \mathbf{-})\rightarrow \mathcal{D}_{(\mathbf{-})} 
\]
inducing an equivalence on $\SymL$ and $\QuadL$ for every polyhedron $X$.
The objects in $\cat{B}(\mathbb{Z},X)$ can be viewed as covariant
functors, i.e. cosheaves over open stars of $X$. The constructed
equivalence is geometric in the sense that it is given by canonically extending a $(\mathbb{Z},X)$-module to a
homotopy cosheaf. 

In the first four sections we recall all the background definitions. In section 2 we clarify what our framework for $L$-theory is going to be. In sections 3 and 4 we collect the definitions of the categories $\cattwo{K}_{(\mathbb{Z},X)} $ resp. $\mathcal{K}_X$. In the remaining sections, the original work is presented. For a fixed simplicial complex $X$ we construct a functor from  $\mathbb{B}(\mathbb{Z},X)$ to $\mathcal{D}_X$ giving rise to the functor $\cattwo{K}_{(\mathbb{Z},X)}\rightarrow\cattwo{K}_X$ in section 5 and prove its naturality. In section 6 we define a natural transformation between the derived products of the latter categories and show that it preserves non-degeneracy. The final theorem is stated in section 7. A few remarks on earlier work on this subject are made in section 8. The final section is an appendix containing some remarks on homotopy (co)limits in the category of chain complexes.
\section*{acknowledgements}
This article is part of the author's PhD thesis written at the University of Aberdeen under the supervision of Michael Weiss. I am indebted to Michael Weiss for sharing his knowledge of $L$-theory. I am also grateful to Andrew Ranicki for enlightening conversations. Last but not least I thank the anonymous referee for pointing out an inconsistency in section 8. 
\section{Remarks on $L$-theory}\label{sec:2}
\subsection{\bf{$L$-theory of additive categories without explicit chain duality}}\label{subsec:2.1}
In \cite{Ra92} a very general framework for $L$-theory is given. An
algebraic bordism category consists of an additive category
$\mathbb{A}$ with chain duality, a subcategory of the category of chain complexes in
$\mathbb{A}$ and a subcategory of ``contractible'' complexes.
Symmetric and quadratic $L$-groups as well as the corrensponding
spectra are defined for every such category. The most natural way of
comparing the constructions of \cite{Ra92} and \cite{RaWei10} would
be to construct a functor of algebraic bordism categories and to
show that it induces an isomorphism on $L$-groups. However, the
structure of an algebraic bordism category is unsuitable for the
homotopy cosheaves of \cite{RaWei10}. It turns out to be difficult
to define a chain duality: the duals are only given implicitly since the
objects are not finitely generated in general. Instead, in
\cite{RaWei10} a slight modified setting is presented. In the
following we will only deal with $L$-theory in this setting.

A chain duality is needed to pass from chain complexes in
$\mathbb{A}$ to chain complexes of abelian groups. When there is a
chain complex  and an action of $\mathbb{Z}/2$ on it, its homotopy
fixed points (resp. homotopy orbits) can be considered and the rest
is \textit{as usual} as $L$-theorists would put it. That is exactly what the formalism of \cite{RaWei10}
establishes by defining a ``chain product''. The crucial properties of a product suitable for doing $L$-theory are extracted in \cite{WW98,WW00}. Therefore it is not surprising that the axioms of a chain product below resemble very much those for an $SW$-product. In our situation however the underlying category is still additive, so in fact the main difference to \cite{Ra92} seems to lie in the lack of an explicit duality. We elaborate on this in the following. 

\begin{definition}
Given an additive category $\mathcal{A}$ consider the category $\mathcal{B}(\mathcal{A})$
of chain complexes in $\mathcal{A}$ bounded from below and from above. Let $\mathcal{C}$ be a full subcategory closed under mapping cones and containing all contractible complexes in $\mathcal{B}(\mathcal{A})$. A complex in $\mathcal{B(A)}$ will be called $\mathcal{C}$-\textit{contractible} if and only if it belongs to $\mathcal{C}$. A morphism in $\mathcal{B(A)}$ will be called a $\mathcal{C}$-\textit{equivalence} or simply \textit{homotopy equivalence} if and only if its mapping cone is $\mathcal{C}$-contractible. Further let $\mathcal{D}$ be a full subcategory of $\mathcal{B}(\mathcal{A})$ closed under
suspension, desuspension, homotopy equivalence, direct sum and
mapping cone. A \textit{chain product}\footnote{This should not be confused with the chain product of \cite[Def. 5.3]{RaWei12}. Our chain product is per definition a bifunctor on chain complexes in $\mathcal{A}$.} on $\mathcal{D}$ is a functor to chain complexes of abelian groups
\[
\mathcal{D}\times\mathcal{D}\rightarrow Ch(\mathcal{A}b),
(C,D)\mapsto C\boxtimes D
\]
satisfying
\begin{enumerate}
\item for $D\in\mathcal{D}$, $C\mapsto C\boxtimes D$ takes
$\mathcal{C}$-contractible objects to contractible ones and preserves homotopy
pushouts,
\item there is a binatural isomorphism $\tau: C\boxtimes
D\rightarrow D\boxtimes C$ and $\tau^2=id$,

\end{enumerate}
The tuple $(\cattwo{A},\mathcal{D},\boxtimes)$ is called an additive category with chain product.
\end{definition}
\begin{remark}
 The category $\mathcal{B}(\mathcal{A})$ is a model category with cofibrations given by valuewise split injections and weak equivalences given by chain homotopy equivalences which are $\mathcal{C}$-equivalences. 
 We will write $\mathcal{HB(A)}$ for the corresponding homotopy category, i.e. for the localisation of $\mathcal{B(A)}$ with respect to $\mathcal{C}$-equivalences. We will write $\mathcal{HD}$ for the correpsonding homotopy category of $\mathcal{D}$ which necessarily becomes a (triangulated) subcategory of $\mathcal{HB(A)}$.
Since homology is homotopy invariant there is an induced bifunctor
\[
\mathcal{HD}\times \mathcal{HD}\rightarrow \mathcal{A}b, \quad (C,D)\mapsto H_0(C\boxtimes D).
\]
\end{remark}
\begin{definition}
In the situation of the previous definition we call $\mathcal{K}=(\mathcal{A},\mathcal{D},\mathcal{C},\boxtimes)$ a \textit{weak algebraic bordism category} if for each $C\in\mathcal{D}$ the functor
\[
D\mapsto H_0(C\boxtimes D)
\]

\[
D\mapsto H_n(C\boxtimes D)
\]
become corepresentable in $\mathcal{HD}$ for all $n$ with corepresenting objects given by the complex $TC[-n]=\Sigma^n TC$.
\end{definition}

\begin{definition}
 An $n$-cycle $\phi$ in $C\boxtimes D$ is called nondegenerate if and only if its adjoint $\Sigma^n TC \rightarrow D$ is a homotopy equivalence.
\end{definition}
There is a $\mathbb{Z}_2$-action on $C\boxtimes C$ via the operator $\tau$. Let $W$ denote the standard free $\mathbb{Z}[\mathbb{Z}_2]$-module resolution of the trivial $\ZZtwo$-module $\mathbb{Z}$. In fact due to well known homological algebra any resolution of $\mathbb{Z}$ by projective $\ZZtwo$-modules is sufficient in the following. 
\begin{definition}
Denote by  $(C\boxtimes C)^{h\mathbb{Z}_2}$ the homotopy fixed points of $C\boxtimes\nolinebreak C$ given by
\[
\Hom_{\ZZtwo}(W,C\boxtimes C).
\] 
If $f:C\rightarrow D$ is a map of chain complexes denote by $f^{h\mathbb{Z}_2}$ the corresponding induced map  $(C\boxtimes C)^{h\mathbb{Z}_2}\rightarrow (D\boxtimes D)^{h\mathbb{Z}_2}$. For a chain $\phi\in(C\boxtimes C)^{h\mathbb{Z}_2}$ we will indicate by $\phi_0$ the projection\footnote{This is given by the image under $\phi$ of the generator 1 in $W_0$} to $C\boxtimes C$. We call a cycle $\phi$ in 
$(C\boxtimes C)^{h\mathbb{Z}_2}$ nondegenerate if and only if $\phi_0$ is nondegenerate.

A symmetric algebraic Poincar\'e complex (SAPC) of dimension $n$ in $\mathcal{D}$ is a pair $(C,\phi)$  with $C$ a chain complex in $\mathcal{D}$ and $\phi$ a nodegenerate cycle in $(C\boxtimes C)^{h\mathbb{Z}_2}$. A symmetric algebraic Poincar\'e pair (SAPP) of dimension $n+1$ is a triple 
\[
(f:C\rightarrow D,\delta\phi,\phi)
\] 
with $f$ a map of chain complexes in $\mathcal{D}$ and $(\delta\phi,\phi)$ a nondegenerate cycle in \linebreak $\Cone(f^{h\mathbb{Z}_2})$
. The last condition means that $\phi$ is an $n$-cycle in $(C\boxtimes C)^{h\mathbb{Z}_2}$, $\delta\phi$ an $(n+1)$-chain in $(D\boxtimes D)^{h\mathbb{Z}_2}$ satisfying $f^{h\mathbb{Z}_2}(\phi)=\partial\delta\phi$, $\phi$ is nondegenerate in $H_n(C\boxtimes C)$ and the image of $\delta\phi$ is nondegenerate in 
$H_{n+1}(D\boxtimes\Cone(f))$.\\\\
Two symmetric algebraic Poincar\'e complexes $(C,\phi)$ and $(C',\phi')$ are called \textit{bordant} if and only if there exist a SAPP $(C,\delta\phi,\phi\oplus -\phi')$.
 \end{definition}
Analogously one can make the following
\begin{definition}
Denote by  $(C\boxtimes C)_{h\mathbb{Z}_2}$ the homotopy orbits of $C\boxtimes C$ given by 
\[
W\otimes_{\ZZtwo}C\boxtimes C.
\]
For a map $f:C\rightarrow D$ of chain complexes write $f_{h\mathbb{Z}_2}$ for the induced map 
\[
(C\boxtimes C)_{h\mathbb{Z}_2}\rightarrow (D\boxtimes D)_{h\mathbb{Z}_2}.
\]
For a chain $\phi\in(C\boxtimes C)_{h\mathbb{Z}_2}$ let $\phi_0$ the projection\footnote{This is given by projecting $\phi$ to $1\otimes \phi_0$ first, where 1 is the generator of $W_0$} to $C\boxtimes C$. A cycle $\phi$ in 
$(C\boxtimes C)_{h\mathbb{Z}_2}$ is called nondegenerate if and only if $(1+\tau)\phi_0$ is nondegenerate.

A quadratic algebraic Poincar\'e complex (QAPC) of dimension $n$ in $\mathcal{D}$ is a pair $(C,\phi)$  with $C$ a chain complex in $\mathcal{D}$ and $\phi$ a nodegenerate cycle in $(C\boxtimes C)_{h\mathbb{Z}_2}$. A quadratic algebraic Poincar\'e pair (QAPP) of dimension $n+1$ is a triple\\ 
\[
(f:C\rightarrow D,\delta\phi,\phi)
\] 
with $f$ a map of chain complexes in $\mathcal{D}$ and $(\delta\phi,\phi)$ a nondegenerate cycle in \linebreak$\Cone(f_{h\mathbb{Z}_2})$, i.e. $\phi$ is an $n$-cycle in $(C\boxtimes C)_{h\mathbb{Z}_2}$, $\delta\phi$ an $(n+1)$-chain in $(D\boxtimes D)_{h\mathbb{Z}_2}$ satisfying $f^{h\mathbb{Z}_2}(\phi)=\partial\delta\phi$, $\phi$ is nondegenerate in $H_n(C\boxtimes C)$ and the image of $\delta\phi$ is nondegenerate in 
$H_{n+1}(D\boxtimes\Cone(f))$.\\\\
Two quadratic algebraic Poincar\'e complexes $(C,\phi)$ and $(C',\phi')$ are called \textit{bordant} if and only if there exist a QAPP $(C,\delta\phi,\phi\oplus -\phi')$.
 \end{definition}

\begin{definition}
 The $n$-dimensional symmetric $L$-groups $L^n(\mathcal{K})=L^n(\mathcal{D})$ of a weak algebraic bordism category $\mathcal{K}=(\cattwo{A},\cattwo{D},\mathcal{C},\boxtimes)$ are defined to be the bordism groups of $n$-dimensional SAPC's in $\cattwo{D}$. The $n$-dimensional quadratic $L$-groups $L_n(\mathcal{K})=L_n(\mathcal{D})$ are the bordism groups of $n$-dimensional QAPC's in $\cattwo{D}$.
\end{definition} 

The following describes a general principle going back to Quinn of interpreting $L$-groups as homotopy groups of certain $L$-spectra. Details can be found in \cite[Ch12,13]{Ra92} and also in \cite{RaWei12}. The framework in \cite{LM12} is more modern and more general.
\begin{proposition}
 In the above setting one can construct an $\Omega$-spectrum of Kan-$\triangle$-sets $\SymL(\mathcal{K})$ out of $m$-ads of $n$-dimensional SAPC's in $\mathcal{D}$ with the property that
\[
\pi_n(\SymL(\mathcal{K}))=L^n(\mathcal{K}).
\]
Similarly there exists an $\Omega$-spectrum of Kan-$\triangle$-sets $\QuadL(\mathcal{K})$ with the propery
\[
\pi_n(\QuadL(\mathcal{K}))=L_n(\mathcal{K}).
\]
\end{proposition}
To compare $L$-groups of different categories we will need the following
\begin{definition}
 Given two weak algebraic bordism categories $\mathcal{K}=(\mathcal{A},\mathcal{D},\mathcal{C},\boxtimes)$ and $\mathcal{K}'=(\mathcal{A}',\mathcal{D}',\mathcal{C}',\boxtimes')$. A functor $F:\mathcal{A}\rightarrow\mathcal{A}'$ is called a \textit{functor of weak algebraic bordism categories} if 
\begin{enumerate}
 \item $F$ is exact in the sense that it preserves cofibrations and weak equivalences, takes $\mathcal{C}$ into $\mathcal{C}'$ and $\mathcal{D}$ into $\mathcal{D}'$,
 \item there exists a natural transformation $h=h_{C,D}:C\boxtimes D\rightarrow F(C)\boxtimes' F(D)$ commuting with the symmetry operator and taking nondegenerate cycles to nondegenerate ones.
\end{enumerate}
\end{definition}
\begin{proposition}\label{An equivariant functor induces maps of spectra}
A functor of weak algebraic bordism categories induces maps of spectra 
\[
F^\bullet:\SymL(\mathcal{K})\rightarrow\SymL(\mathcal{K}') \qquad F_\bullet:\QuadL(\mathcal{K})\rightarrow\QuadL(\mathcal{K}')
\]
and hence maps between corresponding $L$-groups.
\end{proposition}
\begin{proof}
Thinking on the level of the $L$-groups the statement looks obvious, since the natural transformation $h$ implies that $F$ maps Poincar\'e objects to Poincar\'e objects and bordant objects to bordant ones: A SAPC $(C,\phi)$ in $\mathcal{D}$ gives rise to a SAPC $(F(C),h(\phi))$ in $\cattwo{D}'$. If two SAPC's $(C,\phi),\;(C',\phi')$ are bordant via 
\[
(f:C\oplus C',\delta\phi,\phi\oplus-\phi')
\] 
then their images $(C,\phi),\;(C',\phi')$ are bordant via 
\[
(F(f):F(C)\oplus F(C'),h(\delta\phi),h(\phi\oplus-\phi')).
\]
Now to lift this to a spectrum map observe that the assignment $(C,\phi)\mapsto (\mathcal{F},h(\phi))$ respects the gluing constructions of
$[m]$-ads and generalises to a map of $[m]$-ads of SAPC in $\cattwo{D}$ to $[m]$-ads of SAPC in $\cattwo{D}'$, which is well defined because of the above. Hence it gives rise to an induced
map of Kan $\triangle$-sets and corresponding $\Omega$-spectra. The quadratic case is analogous. See also \cite[\S 13]{Ra92} for this sort of reasoning.

\end{proof}
\subsection{\bf{Algebraic Bordism Categories}\label{Algebraic Bordism Categories Basics}}\label{subsec:2.2}
First an elementary observation. Let $\cat{A}$ be an additive category and $\cat{B(A)}$ the category of (bounded) chain complexes in $\cat{A}$. A contravariant additive functor 
\[
T:\cat{A}\rightarrow \cat{B(A)}
\]
can be extended to a contravariant additive functor
\[
T:\cat{B(A)}\rightarrow \cat{B(A)}
\]
simply by taking the total complex of the double complex arising by applying $T$ degreewise. Now the following definition makes sense.
\begin{definition}
 Let $\cat{A}$ be an additive category and $\cat{B(A)}$ the category of chain complexes in $\cat{A}$. Given a contravariant additive functor
$
T:\cat{A}\rightarrow \cat{B(A)}
$
and a natural transformation $e:T^2\rightarrow \id_\cat{A}$ the triple $(\cat{A},T,e)$ is called an \textit{additive category with chain duality} if and only if
\begin{enumerate}
 \item $e(T(A))\circ T(e(A))=\id_{T(A)}$,
 \item $e(A):T^2(A)\rightarrow A$ is a chain equivalence.
\end{enumerate}
\end{definition}
\begin{definition}
Given an additive category with chain duality $(\cat{A}, T, e)$ one defines a product of two objects $M,N\in\cat{A}$ by
\[
M\otimes_\cat{A} N=\Hom_{\cat{A}}(TM,N)
\]
which can be extended to a product of two chain complexes $C,D\in\cat{B(A)}$
\[
C\otimes_\cat{A} N=\Hom_{\cat{A}}(TC,D).
\]
The duality functor $T$ induces a $\mathbb{Z}_2$-action on $C\otimes_\cat{A} C$. The definitions of symmetric and quadratic 
Poincar\'e complexes and pairs in $\cat{B(A)}$ carry over verbatim from above. The symmetric (resp. quadratic) $L$-
groups $L^n(\cat{A})$ (resp. $L_n(\cat{A})$) as bordism groups of SAPC's (resp. QAPC's) in $\cat{B(A)}$.
\end{definition}

Now we can slightly generalise this notion by restricting the choice of chain complexes or allowing $T^2(A)\rightarrow A$ to be a ``weaker'' equivalence.
\begin{definition}
 Let $\cat{A}$ be an additive cateogory. Given a full subcategory $\cat{C}$ of the category of (bounded) chain complexes 
 $\cat{B(A)}$ which is closed under mapping cones, a chain complex $C\in\cat{B(A)}$ is called $\cat{C}$-\textit{contractible} if $C$ is in $\cat{C}$. A chain map $f:C\rightarrow D$ is called a $\cat{C}$-{equivalence} if the mapping cone $\Cone(f)$ is in $\cat{C}$. Assume now $(\cat{A}, T, e)$ is an additive category with chain duality and two subcategories $\cat{B},\cat{C}$ of $\cat{B(A)}$ are specified which are closed under mapping cones and $\cat{C}$ is contained in $\cat{B}$. A triple $\Lambda=(\cat{A},\cat{B},\cat{C})$ is called an \textit{algebraic bordism category} if and only if for each $C\in\cat{B}$
\begin{enumerate}
 \item the mapping cone $\Cone(\id_C)$ is in $\cat{C}$,
 \item the chain equivalence $T^2(C)\rightarrow C$ is a $\cat{C}$-equivalence.
\end{enumerate}
\end{definition}
\begin{definition}
 Given an algebraic bordism category $\Lambda=(\cat{A},\cat{B},\cat{C})$ we can follow the above recipe to define symmetric and quadratic algebraic Poincar\'e complexes respectively pairs in $\cat{B}$. A cycle $\phi\in C\otimes_\cat{A} C$ is considered nondegenerate here if and only if the mapping cone of its adjoint is a $\cat{C}$-equivalence. The symmetric (resp. quadratic) $L$-groups $L^n(\Lambda)$ (resp. $L_n(\Lambda)$) of $\Lambda=(\cat{A},\cat{B},\cat{C})$ are then defined as bordism groups of SAPC's (resp. QAPC's) in $\cat{B}$.
\end{definition}
\begin{example}\label{Algebraic Bordism Category of a ring with involution}

Let $R$ be a ring with involution $\iota$. Let $\cat{A}(R)$ be the category of f.g. projective left $R$-modules. Define $T:\cat{A}(R) \rightarrow \cat{A}(R)\subset\cat{B}(\cat{A}(R))$ by mapping a module $M$ to $M^{\ast t}=Hom^t_R(-,R)$ where the superscript $t$ indicates that the right $R$-module $M^\ast$ is viewed as a left module via the involution. Let $\cat{B}(R)$ consist of finite chain complexes of f.g. projective left $R$-modules and $\cat{C}(R)$ of contractible ones. This defines an algebraic bordism category $\Lambda=(\cat{A}(R),\cat{B}(R),\cat{C}(R))$ and the symmetric (resp. quadratic) $L$-groups $L^n(\Lambda)$ (resp. $L_n(\Lambda$) are the symmetric (resp. quadratic) $L$-groups of the ring with involution $R$. For the group ring $R=\mathbb{Z}[\pi]$ and the canonical involution the quadratic groups $L_n(\mathbb{Z}[\pi])$ are the (projective) surgery obstruction groups of Wall. The symmetric groups are the nonperiodic versions of symmetric $L$-groups of Mishchenko.
 \end{example}
Here is the main example of a weak algebraic bordism category.
\begin{example}
\label{Algebraic Bordism Category gives rise to a weak algebraic bordism category}
Given an algebraic bordism category
$\Lambda=(\mathbb{A},\mathbb{B},\mathbb{C})$ such that $\mathbb{C}$ contains chain contractible complexes in $\mathbb{B}$ the category
$\mathcal{K}_\Lambda=(\mathbb{A},\mathbb{B},\mathbb{C},\boxtimes_T)$ is a weak algebraic bordism category where
$C\boxtimes_{T}D:=C\otimes D$
\end{example}
\begin{remark}
There is a notion of a functor of algebraic bordism categories. Such a functor induces maps of $L$-spectra and $L$-groups. We will not make use of this notion here. It is however important to notice that such a functor gives rise to a functor of the corresponding weak algebraic bordism categories. This will be used later.
\end{remark}

\section{The algebraic bordism category $\cat{A}(\mathbb{Z},X)$}\label{sec:3}
Let $X$ be a simplicial complex. In this section we recall the
definition of the $\cat{A}(\mathbb{Z},X)$. The reference is
\cite[\S 4 et seq.]{Ra92}.
\begin{definition}
\hfill
\begin{enumerate}
\item
Let $\mathbb{A}$ be an additive category. An object $M\in \cat{A}$
is $X$-based if it is expressed as a direct sum
\[
M=\sum_{\sigma\in X}M(\sigma)
\]
of objects $M(\sigma)\in \cat{A}$ s.t. $\{\sigma\in X \mid
M(\sigma)\neq0\}$ is finite. A morphism $f:M \rightarrow N$ of $X$-based
objects is a collection of morphisms in $\cat{A}$
\[
\{f(\tau,\sigma):M(\sigma)\rightarrow N(\tau)\mid\sigma,\tau\in X\}.
\]
\item
Let $\cat{A}_*(X)$ be the additive category of $X$-based objects $M$
 with morphisms $f:M\rightarrow N$ s.t.
 $f(\tau,\sigma):M(\sigma)\rightarrow N(\tau)$ is $0$ unless
 $\tau\geq\sigma$ so that
 \[
f(M(\sigma))\subseteq\sum_{\tau\geq\sigma} N(\tau).
 \]
\item
Forgetting the $X$-based structure defines the covariant
\textit{assembly} functor
\[
\cat{A}_*(X)\rightarrow\cat{A}, \quad M\mapsto
M^*(X)=\sum_{\sigma\in X}M(\sigma).
\]
\end{enumerate}
\end{definition}
\begin{definition}
A $(\mathbb{Z},X)$-module is an $X$-based object in
$\mathbb{A}(\mathbb{Z})$, where $\mathbb{A}(\mathbb{Z})$
denotes the additive category of free abelian groups. 
\end{definition}
\begin{remark}\label{Generation by free (Z,X)-modules}
 A \textit{free} $(\mathbb{Z},X)$-module on one generator $M_\sigma$ is given by
\[
M_\sigma(\tau)=\left\{
\begin{array}{ll}
 \mathbb{Z}&\sigma=\tau\\
0 & \sigma\neq\tau\\
\end{array}
\right.
\]
for some simplex $\sigma\in X$. Every $(\mathbb{Z},X)$-module is a direct sum of free $(\mathbb{Z},X)$-modules on one generator. 
\end{remark}

Here and in the following $\triangle_\ast(X)$ stands for the simplicial chain complex of a simplicial complex $X$.
\begin{example}
The simplicial cochain complex $\triangle(X)^{-*}$ of $X$ is a
finite chain complex in $\cat{A}(\mathbb{Z})_*(X)$ with
\[
\triangle(X)^{-*}(\sigma)= S^{-\md{\sigma}}\mathbb{Z}.
\]
\end{example}
\begin{definition}
\hfill
\begin{enumerate}
\item Let $\cat{A}_*[X]$ be the category with objects the
contravariant additive functors
\[
M: X \rightarrow \cat{A}, \quad \sigma \mapsto M[\sigma]
\]
s.t. $\{\sigma\mid M[\sigma]\neq 0\}$ is finite. The morphisms are
natural transformations of such functors. Here we view $X$ as a
category consisting of simplices and face inclusions.
\item 
We have a covariant functor
\[
\cat{A}_\ast(X)\rightarrow\cat{A}_\ast[X],\quad M\mapsto [M],
[M][\sigma]=\sum_{\tau\geq\sigma}M(\tau)
\]
\end{enumerate}
\end{definition}
\begin{remark}\label{Generation by representables}
 The assembly functor embeds $\cat{A}_\ast(X)$ into $\cat{A}_\ast[X]$ as a full subcategory. Furthermore, every object in $\cat{A}_\ast[X]$ can be obtained by taking (valuewise) direct sum of functors of the form $M_{[\sigma]}$ where the latter is the free abelian group generated by $\Hom_{X}(-,\sigma)$. We can use Remark \ref{Generation by free (Z,X)-modules} to identify $M_{[\sigma]}=[M_\sigma]$.
\end{remark}

\begin{remark}
 To simplify notation we will sometimes write $M$ for $[M]$ in the hope that no confusion is caused. This is in particular reasonable when the type of brackets around the argument determines whether $M$ is in $\mathbb{A}_\ast(X)$ or $\mathbb{A}_\ast[X]$:  $M(\sigma)$ and $M[\sigma]=[M][\sigma]=\sum_{\tau\geq\sigma}M(\sigma)$.
\end{remark}

\begin{example}\label{guiding example of (Z,X)-modules}
Given a simplicial complex $Y$ denote by $D(\sigma,Y)$ the dual cell of $\sigma$ and by $\partial D(\sigma,Y)$ its boundary i.e. the union of dual cells of simplices having $\sigma$ as a proper face. A simplicial map $f:Y\rightarrow X$ gives rise to a complex of $(\mathbb{Z},X)$-modules $C_f$ defined as
\[
C_f(\sigma)=\triangle_\ast(f^{-1}D(\sigma,Y),f^{-1}\partial D(\sigma)).
\]
We have 
\[
[C_f][\sigma]=\sum_{\tau\geq\sigma}C_f(\tau)=\triangle_\ast(f^{-1}D(\sigma,Y)).
\]
Its assembly equals $\triangle_\ast(Y')$ -- the simplicial chain complex of the barycentric subdivision of $Y$.
\end{example}

\begin{proposition}{\cite[5.1]{Ra92}}\label{Lambda and X give rise to Lambda(X)}
Given an algebraic bordism category $\Lambda$ and a locally finite
simplicial complex $X$. The chain duality functor of $\Lambda$ induces a chain duality functor on $\cat{A}_\ast(X)$. Let $\cat{B}_\ast(X)$ be the category of chain complexes $B$ in
$\cat{B}(\cat{A}_\ast(X))$ such that $B(\sigma)$ is in $\cat{B}$ for
every $\sigma$ and similarly for $\cat{C}_\ast(X)$. This makes the triple
$\Lambda_\ast(X)=(\cat{A}_\ast(X),\cat{B}_\ast(X),\cat{C}_\ast(X))$ into an algebraic bordism category,
\end{proposition}
\begin{proposition}\cite[5.6]{Ra92}
 A simplicial map $f:X\rightarrow Y$ of finite (ordered) simplicial complexes induces a functor of algebraic bordism categories
\[
f_\ast: \Lambda_\ast(X)\rightarrow \Lambda_\ast(Y)
\]
determined by the assignment $f_\ast M(\sigma)=\sum\limits_{\substack{\tau\in X \\ f\tau=\sigma}}M(\tau)$
\end{proposition}

\begin{remark}
If $\cat{C}_\ast(X)$ contains all contractible complexes in $\cat{B}_\ast(X)$, the above algebraic bordism category gives rise to a weak algebraic bordism category according to Ex. \ref{Algebraic Bordism Category gives rise to a weak algebraic bordism category}. 
\end{remark}
Let $R$ in Ex. \ref{Algebraic Bordism Category of a ring with involution} be $\mathbb{Z}$ with the trivial involution and consider now the corresponding algebraic bordism category of \textit{free} $\mathbb{Z}$-modules $\Lambda(\mathbb{Z})=(\cat{A}(\mathbb{Z}),\cat{B}(\mathbb{Z}),\cat{C}(\mathbb{Z}))$. 
There is an algebraic bordism category $\Lambda(\mathbb{Z},X)=(\mathbb{A}(\mathbb{Z},X), 
\mathbb{B}(\mathbb{Z},X),\mathbb{C}(\mathbb{Z},X))$ defined as $\Lambda(\mathbb{Z})_\ast(X)=(\mathbb{A}(\mathbb{Z})_\ast(X), 
\mathbb{B}(\mathbb{Z})_\ast(X),\mathbb{C}(\mathbb{Z})_\ast(X))$ in Prop. \ref{Lambda and X give rise to Lambda(X)} which due to the last remark defines the weak algebraic bordism category we will be dealing with in later sections.
\begin{definition}
For the algebraic bordism category $(\mathbb{A}(\mathbb{Z},X), 
\mathbb{B}(\mathbb{Z},X),\mathbb{C}(\mathbb{Z},X))$ of
$(\mathbb{Z},X)$-modules let 
\[
\catZXmod=(\mathbb{A}(\mathbb{Z},X),\mathbb{B}(\mathbb{Z},X),\mathcal{C}(\mathbb{Z},X),\boxtimesRA)
\]
denote the corresponding weak algebraic bordism category of
$(\mathbb{Z},X)$-modules. The chain product is given by
\[
M\boxtimesRA\! N= \Hom_{\mathbb{A}_\ast(X)}(TM,N)=([M]\otimes[N])_\ast[K]=\colim[\sigma\in X][M][\sigma]\otimes [N][\sigma]
\]
\end{definition}

\begin{theorem}{\cite[\S 13]{Ra92}}\label{X to (Z,X)-modules is homology theory}
The symmetric (resp. quadratic) $L$-groups $L^n(\cattwo{K}_{(\mathbb{Z},X)})$ \!\!(resp. $L_n(\cattwo{K}_{(\mathbb{Z},X)})$) can be
identified with the $\SymL$-homology groups $H_n(X;\SymL(\mathbb{Z}))$ (resp. $\QuadL$-homology groups $H_n(X;\QuadL(\mathbb{Z})))$.
\end{theorem}
\begin{definition}\label{L-fundamental class (Z,X)-modules}
Let $X$ be a simplicial complex such that its realisation is an $n$-dimensional closed manifold. Consider the $(\mathbb{Z},X)$-module chain complex $C_{id_X}$ of Ex. \ref{guiding example of (Z,X)-modules}. Over each simplex there is a refinement of the Alexander-Whitney diagonal approximation (cf. \cite[\S 6]{Ra92})
\[
[C_{id_X}][\sigma]\rightarrow ([C_{id_X}][\sigma]\otimes_\mathbb{Z}[C_{id_X}][\sigma])^{h\mathbb{Z}/2}
\]
which fit together to give a map 
\[
[C_{id_X}][X]\rightarrow ([C_{id_X}][X]\otimes_\mathbb{Z}[C_{id_X}][X])^{h\mathbb{Z}/2}.
\]
The image of the fundamental class of $X$ under this is a nondegenerate cycle $\phi_X$. The pair $(C_{id_X},\phi_X)$ is a SAPC in $\cattwo{K}_{(\mathbb{Z},X)}$ and defines a canonical class in $L^n(\cattwo{K}_{(\mathbb{Z},X)})$ which will be denoted by $[X]_{\cattwo{K}_{(\mathbb{Z},X)}}$. If the realisation of $X$ is an $n$-dimensional manifold with boundary, there is an analogously constructed canonical relative class $[X]_{\cattwo{K}_{(\mathbb{Z},X)},\cattwo{K}_{(\mathbb{Z},\partial X)}}$ in $L^n(\cattwo{K}_{(\mathbb{Z},X)},\cattwo{K}_{(\mathbb{Z},\partial X)})$. 
\end{definition}

\section{Ranicki-Weiss Cosheaves}\label{sec:4}
The objects of this section are constructed in \cite{RaWei10}. While the $(\mathbb{Z},X)$-module chain complexes can be viewed as chain complexes of $\mathbb{Z}$-modules labeled by open stars of a simplicial complex, the cosheaves of Ranicki and Weiss are labeled by open sets of a given (ENR) topological space. The main guiding example is Ex. \ref{guiding example}. Dual cells are replaced by open subsets while the simplicial chain complex is replaced by the singular one. The analogous condition for a $\mathbb{Z}$-module of being $X$-based is expressed in the next definition.
\begin{definition}\label{definition O(X)-modules}
Let $X$ be a locally compact, Hausdorff and separable space and write $\Open(X)$ for the category of open sets of $X$. Let $\mathrm{F}$ be a free abelian group with a basis $B$. We call $\mathrm{F}$ $\Open(X)$-\textit{based} if and only if there is a covariant functor $F:\Open(X)\rightarrow \mathcal{A}b$ to abelian groups such that 
\begin{enumerate}
\item $F(\emptyset)=0,\quad F(X)=\mathrm{F}$,
 \item $F(U)$ is generated by a subset $B_U$ of $B$,
 \item for $U,V \in\Open(X),\quad F(U\cap V)=F(U)\cap F(V)$.
\end{enumerate}
A morphism between two $\Open(X)$-based abelian groups is a group homomorphism $f:\mathrm{F}_0\rightarrow\mathrm{F}_1$ taking $F_0(U)$ to $F_1(U)$ for every open set $U\in\Open(X)$. Denote by $\cattwo{A}=\cattwo{A}_X$ the additive category of $\Open(X)$-based groups over $X$.
\end{definition}
\begin{example}
 For any $i\geq 0$ let $\mathrm{S}$ be the $i$-th singular chain group of $X$ $S_i(X)$ with $B$ consisting of the singular $i$-simplices in $X$. Since for $U\in\Open(X)$ the subgroup $S_i(U)$ of $\mathrm{S}$ is generated by $i$-simplices in $X$ with image in $U$, it is obvious that $\mathrm{S}$ is $\Open(X)$-based.
\end{example}

\begin{definition}
Let $\cattwo{B}(\cattwo{A})$ denote the category of chain complexes
in $\cattwo{A}$ which are bounded from below.
\end{definition}

\begin{example}\label{guiding example}
Let $f: Y\rightarrow X$ be a map from a compact ENR $Y$. Define an
object $C(f)\in\cattwo{B}(\cattwo{A})$ by
$C(f)(X)=S_\ast(f^{-1}(X))$ the singular chain complex of $X$ with
the standard basis and for $U\in\Open(X), C(f)(U)\subset C(f)(X)$
the subcomplex generated by simplices with image in $f^{-1}(U)$.
\end{example}
\begin{definition}\label{sheaf condition}
\hfill
\begin{enumerate}
\item An object $C\in\cattwo{B}(\cattwo{A})$ satisfies the
\textit{sheaf type condition} if for any
$\mathcal{W}\subset\Open(X)$ the inclusion
\[
\sum_{V\in\mathcal{W}}C(V)\rightarrow C(\bigcup_{V\in\mathcal{W}}V)
\]
is a homotopy equivalence, where the sum on the left is taken inside
$C(X)$.
\item An object $C\in\cattwo{B}(\cattwo{A})$ satisfies \textit{finiteness condition (i)} if there exists an integer $a\geq
0$ such that: for every inclusion of open sets $V_1\subset V_2$ with
$\overline V_1\subset V_2$, the induced inclusion $C(V_1)\subset
C(V_2)$ factors up to homotopy through a complex $D$ of finitely
generated free abelian groups, bounded by $a$ from above and from
below.
\item An object $C\in\cattwo{B}(\cattwo{A})$ satisfies \textit{finiteness condition (ii)}
if there exists a compact subset $K$ of $X$ such that $C(U)$ depends
only on $C(U\cap K)$. In this case, $C$ is said to be supported in
$K$.
\end{enumerate}
We write $\mathcal{C}$ for the subcategory of
$\cattwo{B}(\cattwo{A})$ consisting of objects satisfying all the
above conditions. As usual, we write $\mathcal{C}_X, \mathcal{C}_Y$
etc. to emphasize the dependance on the space.
\end{definition}

\begin{remark}
The example \ref{guiding example} satisfies all three conditions of
the last definition.
\end{remark}
\begin{definition}\label{Pushforward of cosheaves}
A map $f:X\rightarrow Y$ induces a (covariant) pushforward functor
\[
f_\ast:\mathcal{C}_X\rightarrow \mathcal{C}_Y
\] 
defined by $f_\ast C(U)=C(f^{-1}(U))$. 
\end{definition}
\begin{lemma}\cite[3.9, 3.10]{RaWei10}\label{homotopy cosheaf condition}
 Let $C$ be in $\mathcal{C}_X$ and $\mathcal{W}$ be a subset of $\Open(X)$.
\begin{enumerate}
 \item[a)] If $\mathcal{W}$ is finite and closed under unions, inclusions induce a homotopy equivalence
\[
C(\bigcap_{V\in\mathcal{W}}V) \xrightarrow{\simeq} \holim[V\in\mathcal{W}]C(V).
\] 
 \item[b)] If $\mathcal{W}$ closed under intersections, the inclusions induce a homotopy equivalence
\[
\hocolim[V\in\mathcal{W}]C(V)\xrightarrow{\simeq}C(\bigcup_{V\in\mathcal{W}}V).
\]
\end{enumerate}

\end{lemma}
\begin{remark}
Let $F$ be a contravariant functor $\Open(X)\rightarrow\mathcal{H}$ with a notion of homotopy in the target category. Following the general principle of taking the homotopy limit instead of the ordinary one, $F$ is called a homotopy sheaf in the literature if for every $W\in\Open(X)$
\[
F(W)\rightarrow \holim[I]F(V_I)
\]
is a homotopy equivalence where $I$ runs through finite intersections of open sets covering $W$. Property b) of the preceding lemma is dual to this criterion. Therefore it seems consistent to call objects satisfying Def. \ref{sheaf condition} and hence Lemma \ref{homotopy cosheaf condition} b) \textit{homotopy cosheaves} of chain complexes.
\end{remark}
\begin{definition}
Denote by $\mathcal{C}''$ the full subcategory of objects for which $C(U)$ is
contractible for all $U\in\Open(X)$. A morphism $f:C\rightarrow D$ in $\mathcal{C}$ is called
\textit{weak equivalence} if its mapping cone belongs to
$\mathcal{C}''$
\end{definition}
\begin{remark}
With the chain product defined below the tuple $(\cattwo{A}_X,\cattwo{C}_X,\cattwo{C}''_X,\boxtimes)$ defines a weak algebraic bordism category. The corresponding functor from $X$ to $L$-theory of this does not satisfy excision though. To resolve this, a full subcategory of $\cattwo{C}$ is introduced in the following.
\end{remark}
\begin{definition}\label{Definition:well generated}
Let $\mathcal{D}$ be the smallest full subcategory of $\mathcal{C}$
satisfying the following.
\begin{enumerate}
\item
All objects of $\mathcal{C}$ obtained from $f:\triangle^k\rightarrow
X$ as in example \ref{guiding example} are in $\mathcal{D}$.
\item
If two of three objects in a short exact sequence $C\rightarrow
D\rightarrow E$ are in $\mathcal{D}$ then is the third.
\item
All weakly contractible objects are in $\mathcal{D}$, i.e.
$\mathcal{C}''\subset\mathcal{D}$.
\end{enumerate}
\end{definition}
\begin{remark}
For a map $f:X\rightarrow Y$ and $C\in\cattwo{D}_X$ we have $f_\ast C\in\cattwo{D}_Y$.
\end{remark}
\begin{definition}
Given two objects $C,D\in\mathcal{C}$.  Define their chain product by
\[
C\boxtimes D=\holimtwo{U\subset X \textrm{open}, K_1,K_2\subset X
\textrm{closed}}{K_1\cap K_2\subset U} C(U,U\setminus
K_1)\otimes_{\mathbb{Z}}D(U,U\setminus K_2)
\]
where the values of $C$ resp. $D$ on pairs are defined in the usual
way as quotients.
\end{definition}
\begin{remark}
A map $f:X\rightarrow Y$ induces a map of products
\[
C\boxtimes D\xrightarrow{f^{\boxtimes}}f_\ast C\boxtimes f_\ast D
\]
given by projections (i.e. specialisation to open sets in the preimage of $f$).
\end{remark}
\begin{proposition}\cite[7.3]{RaWei10}
 Mapping $X$ to $\cattwo{D}_X$ is functorial and preserves duality, i.e. if $\phi\in C\boxtimes D$ is nondegenerate its image $f^\boxtimes(\phi)$ is also nondegenerate.
\end{proposition}
\begin{remark}
 Our exposition here is significantly shorter than in the original\linebreak source \cite{RaWei10}. A large part there is devoted to decomposability of $\cattwo{D}$, which is crucial for the excision property of $X\rightarrow \SymL(\cattwo{K}_X)$. Another issue to mention is that the authors do not work with the homotopy category $\mathcal{HD}$. Instead, they introduce the subcategory $\cattwo{D}'$ of \textit{free} objects, closed under taking duals, and show that every complex in $\cattwo{D}$ can be resolved by one in $\cattwo{D}'$. For the purpose of $L$-theory, this amounts to the same as working in $\cattwo{D}$ but replacing the homotopy category $\mathcal{HD}$ by the localisation with respect to the bigger class of weak equivalences  ($\cattwo{C}''$-equivalences), i.e. defining corepresentability and nondegeneracy by means of $\cattwo{C}''^{-1}\mathcal{D}$ instead of $\mathcal{HD}$.
\end{remark}

\begin{definition}
 We write $\cattwo{K}_X$ for the weak algebraic bordism category \linebreak $(\cattwo{A}_X,\cattwo{D}_X,\cattwo{C}''_X,\boxtimes)$ and 
$\SymL(\cattwo{K}_X)$ (resp. $\QuadL(\cattwo{K}_X)$) for the corresponding $L$-theory spectra. Similar for $L$-groups.
\end{definition}

\begin{theorem}\cite[section 8]{RaWei10}\label{X to cosheaves is homology theory}
\quad\\The covariant functor $X\mapsto \SymL(\mathcal{K}_X)$
satisfies homotopy invariance and excision.
\end{theorem}
\begin{definition}\label{L-fundamental class for cosheaves}
 Let $X$ be a closed $n$-dimensional manifold. As in Ex. \ref{guiding example} its singular chain complex $S_\ast(X)$ defines the complex $C(id_X)$ in $\cattwo{D}_X$. There is a refinement of the Alexander-Whitney map (cf. \cite[Ex. 5.6, 5.9]{RaWei10}) 
\[
S_\ast(X)\rightarrow (C(id_X)\boxtimes C(id_X))^{h\mathbb{Z}/2}
\]
such that the fundamental class of $X$ is mapped to a nondegenerate cycle $\phi_X$. The pair $(C(id_X),\phi_X)$ is a SAPC in $\cattwo{K}_X$ and defines a canonical class in $L^n(\cattwo{K}_X)$ denoted by $[X]_{\cattwo{K}_X}$. Analogously, a $n$-dimensional compact manifold with boundary $(X,\partial X)$ defines a SAPP $(C(id_{\partial X})\hookrightarrow C(id_X),\phi_X,\phi_{X,\partial X})$ and thus defines a canonical relative class $[X,\partial X]_{\cattwo{K}_X}$ in $L^n(\cattwo{K}_X,\cattwo{K}_{\partial X})$.
\end{definition}

\section{The functor $\mathcal{F}$}\label{sec:5}
Let $C$ be in $\cat{B}(\mathbb{Z},X)$. For the corresponding contravariant functor $[C]\in\cat{B}[\mathbb{Z},X]$ it is natural to define a covariant functor $\tilde{C}$ on unions of open stars,
which sends $U=\bigcup\opst{\sigma}$ to 
\[
\colim[\tau,\opst{\tau}\subset U] C[\tau].
\] 
The idea of the following definition is to extend $\tilde{C}$ to a functor in $\cattwo{B}(\cattwo{A}_X)$. We will usually write $C$ for $[C]$.
\begin{definition}\label{THE functor}
We define a functor from the category $\cat{B}(\mathbb{Z},X)$ of $(\mathbb{Z},X)$-module chain complexes to the Ranicki-Weiss category $\cattwo{B}(\cattwo{A}_X)$ of chain complexes labeled by sets in $\Open(X)$ as
\[
\mathcal{F}:\cat{B}(\mathbb{Z},X)\rightarrow\cattwo{B}(\cattwo{A}_X),\quad C\mapsto \mathcal{F}(C): U\mapsto\coend{\sigma}C[\sigma]\otimes S_*(U\cap\sigma)
\]
where $\coend{\sigma}C[\sigma]\otimes S_*(U\cap\sigma)$ is the coend of the functor
\[
SC(U):K^{\textit{op}}\times
K\rightarrow\ChAb,(\sigma,\tau)\mapsto C[\sigma]\otimes
S_*(U\cap\tau)
\]
and the latter is the singular chain complex of $U\cap\tau$. 
\end{definition}
\begin{lemma}
 In the above definition the functor sending $(\sigma,\tau)$ to $C[\sigma]\otimes S_\ast(U\cap\tau)$ is Reedy cofibrant and hence its coend is a model for the homotopy coend, i.e. there is a natural weak equivalence
\[
\coend{\sigma}C[\sigma]\otimes S_*(U\cap\sigma)\simeq \hocoend{\sigma}C[\sigma]\otimes S_*(U\cap\sigma).
\]
\end{lemma}

\begin{remark}\label{F as global coend}
The functor $\mathcal{F}$ can be expressed as a global coend in the following way. View a simplex $\sigma\in X$ as a topological space and let $K(\sigma)\in\mathcal{C}_\sigma$ denote the canonical complex $C(id_\sigma)$ as given in Ex. \ref{guiding example}. Denote by $\iota_\sigma$ the inclusion of the topological space  $\real{\sigma}$ into the realisation of $X$. We have
\[
\mathcal{F}(C)=\coend{\sigma}C[\sigma]\otimes\iota_{\sigma\ast}
K(\sigma)
\]
\end{remark}

\begin{remark}\label{coend as big sum}
Sometimes it is convenient to have the following description of $\mathcal{F}(C)$. In every degree $k$ we have
\[
\mathcal{F}(C)(U)_k=\bigoplus_{i+j=k}\bigoplus_{\sigma\in X}C(\sigma)_i\otimes S_j(U\cap\sigma).
\]
The decomposition of an $X$-based chain complex is not respected by its differentials, however we have
\[
d_k:\bigoplus_{i+j=k}\bigoplus_{\sigma\in X}C(\sigma)_i\otimes S_j(U\cap\sigma)\rightarrow \bigoplus_{\tilde{i}+\tilde{j}=k-1}\bigoplus_{\sigma\in X}C(\sigma)_{\tilde{i}}\otimes S_{\tilde{j}}(U\cap\sigma),
\]
\[
C(\sigma)_i\otimes S_j(U\cap\sigma)\rightarrow\bigoplus_{\tau\geq\sigma}C(\tau)_{i-1}\otimes S_j(U\cap\tau).
\]
\end{remark}

The next lemma shows that the functor $\mathcal{F}(C)$ is
consistent with $C$ and is indeed a (homotopy) extension of $\tilde{C}$.
\begin{lemma}\label{new is old}
If $U$ is a union of open stars $\opst{\sigma}$, then
$\mathcal{F}(C)(U)$ is naturally homotopy equivalent to $\tilde{C}(U)$. If
$(U,V)$ is a pair of unions of open stars then $\mathcal{F}(C)(U,V)$
is naturally homotopy equivalent to $\tilde{C}(U,V)$.
\end{lemma}
\begin{proof}
Let us show the lemma for one open star first, i.e.
\[
\coend{\sigma}\SC{\sigma}{\opst{\tau}}\cong\colim[(\sigma\rightarrow\rho)\in X^\natural]C[\rho]\otimes S_\ast(\sigma\cap\opst{\tau})\simeq C[\tau].
\]
Observe that $\sigma\cap\opst{\tau}$ is non-empty if and only if $\sigma\geq\tau$ and in the latter case we have
\[
S_*(\sigma\cap\bigcup_{\rho\geq\tau}\mathring{\rho})=S_*(\bigcup_{\rho\geq\tau}\sigma\cap\mathring{\rho})\simeq
\sum_{\rho\geq\tau}S_*(\sigma\cap\mathring{\rho}).
\]
The right hand side is naturally chain homotopic to $\mathbb{Z}$ because every
summand clearly is and the sum is taken inside the singular chain complex
$S_*(\opst{\tau})$ of a contractible space $\opst{\tau}$. 
The above diagram satisfies the following property. If $\alpha\cap\opst{\tau}=\emptyset$ the value $C[\beta]\otimes S_\ast(\alpha\cap \opst{\tau})$ at any $\alpha\rightarrow\beta\in X^\natural$ is zero. Furthermore for all maps $(\tilde{\alpha}\rightarrow\tilde{\beta})\rightarrow (\alpha\rightarrow\beta)$ the value at the source $C[\tilde{\beta}]\otimes S_\ast(\tilde{\alpha}\cap \opst{\tau})$ is zero as well since $\tilde{\alpha}\cap\opst{\tau}=\emptyset$. As a consequence the terms at $\alpha\rightarrow\beta$ with $\alpha\cap\opst{\tau}=\emptyset$ can be ignored when taking the colimit, i.e. $\{\alpha\rightarrow\beta\mid\alpha\cap\opst{\tau}\neq\emptyset\}$ is cofinal. Thus we have 
\[
\coend{\sigma}\SC{\sigma}{\opst{\tau}}\cong\colimtwo{(\sigma\rightarrow\rho)\in X^\natural}{\textnormal{s.t. }\tau\leq\sigma}C[\rho]\otimes S_\ast(\sigma\cap\opst{\tau}).
\]
Since the subdiagram is (still) Reedy cofibrant the colimit is actually a homotopy colimit and we may write
\[
\coend{\sigma}\SC{\sigma}{\opst{\tau}}\simeq\hocolimtwo{(\sigma\rightarrow\rho)\in X^\natural}{\textnormal{s.t. }\tau\leq\sigma}C[\rho]\otimes \mathbb{Z}\cong\hocolim[\tau\leq\sigma\in X]C[\sigma]\simeq C[\tau]
\]
where the last step is clear by cofinality. Let $U$ be a union of open stars $U_i=\opst{\tau_i}$. Due to the subsequent lemmata $\mathcal{F}$ takes values in $\cattwo{C}_X$ and we may use Lemma \ref{homotopy cosheaf condition}b) to conclude
\[
\mathcal{F}(C)(\bigcup_{i}U_i)\simeq\hocolim[U_i]\mathcal{F}(C)(U_i)
\]
where the collection of the $U_i$ is
closed under intersections, since the intersection of open stars is
an open star itself. Now each $\mathcal{F}(C)(U_i)$ is naturally homotopy
equivalent to $\tilde{C}(U_i)$ and we may write
\[
\hocolim[U_i]\mathcal{F}(C)(U_i)\simeq\hocolim[U_i\subset U]\tilde{C}(U_i)\simeq \hocolimtwo{\tau_i}{\opst{\tau_i}\subset U}C[\tau_i]
\simeq \tilde{C}(U)
\]
where the last equivalence is due to the fact that $\colim[\tau, \opst{\tau}\subset U] C[\tau] $ computes the homotopy colimit.
\end{proof}
\begin{lemma}
Given $C\in\cat{B}(\mathbb{Z},X)$. Its image $\mathcal{F}(C)$  under $\mathcal{F}$ satisfies the sheaf type condition.
\end{lemma}
\begin{proof}
We have to show that for every subset $\mathcal{W}$ of $\Open(X)$
the inclusion
\[
\sum_{V\in\mathcal{W}} \coend{\sigma}\SC{\sigma}{V}\rightarrow
\coend{\sigma}\SC{\sigma}{\bigcup_{V\in\mathcal{W}}}
\]
is a homotopy equivalence. For every fixed $\sigma$ the inclusion
\[
\sum_{V\in\mathcal{W}} \SC{\sigma}{V}\rightarrow
\SC{\sigma}{\bigcup_{V\in\mathcal{W}}}
\]
is certainly a natural homotopy equivalence due to excision of the singular
chain complex functor $S_*$ (c.f eg. \cite[III, Prop 7.3]{Do80}.
Since the coends compute homotopy coends here, the induced map
\[
\coend{\sigma}\sum_{V\in\mathcal{W}} \SC{\sigma}{V}\rightarrow
\coend{\sigma}\SC{\sigma}{\bigcup_{V\in\mathcal{W}}}
\]
is also a homotopy equivalence. We have to convince ourselves that
the (homotopy) coend and the internal sum of subcomplexes sitting
inside 
\[
\coend{\sigma}\SC{\sigma}{X}
\] 
commute. An analogue of Lemma \ref{homotopy cosheaf condition}b) shows that 
\[\hocolim[V\in
\mathcal{W}]\SC{\sigma}{V}
\] is naturally homotopy equivalent to
$\SC{\sigma}{\bigcup_{V\in\mathcal{W}}}$ and since the (homotopy)
coend is just a (homotopy) colimit, the interchange of sum and coend
follows from a Fubini-like theorem for (homotopy) colimits.
\end{proof}
\begin{lemma}
For $C\in\cat{B}(\mathbb{Z},X)$, $\mathcal{F}(C)$ satisfies the finiteness conditions
$i)$ and $ii)$ of Def. \ref{sheaf condition}.
\end{lemma}
\begin{proof}
$\,$\\
i) Let $V_1,V_2$ be open sets in $X$ such that
$\overline{V}_1\subset V_2$. We have to show that
\[
\coend{\sigma}\SC{\sigma}{V_1}\rightarrow
\coend{\sigma}\SC{\sigma}{V_2}
\]
factors up to chain homotopy through a bounded chain complex of f.g.
free abelian groups. Since $X$ is a simplicial complex we can find a
simplicial subcomplex $Z$ such that $V_1\rightarrow V_2$ factors up
to homotopy through $Z$. 
Passing to the singular chain complex we get a factorization up to
chain homotopy

\[
\begindc{0}[40]
\obj(1,2)[1]{$S_*(\sigma\cap V_1)$} \obj(3,1)[2]{$S_*(\sigma\cap Z)$}
\obj(5,2)[3]{$S_*(\sigma\cap V_2)$}
\mor{1}{2}{}\mor{2}{3}{}[1,0]\mor{1}{3}{}[1,0]
\enddc
\]
where the bottom term is chain homotopy equivalent to the simplicial
chain complex $\triangle_{\ast}(\sigma\cap Z)$, which is a f.g. complex of
free abelian groups. Since $C[\sigma]$ is f.g. and degreewise free
we get a factorization up to chain homotopy
\[
\begindc{0}[4]
\obj(10,23)[A]{$\SC{\sigma}{V_1}$}
\obj(28,10)[B]{$C[\sigma]\otimes\triangle_{\ast}(\sigma\cap Z)$}
\obj(48,23)[C]{$\SC{\sigma}{V_2}$} \mor{A}{B}[15,14]{}
\mor{B}{C}[15,14]{} \mor{A}{C}{}
\enddc
\]
and, applying coend, the desired result.\\\\
ii) We have to show that there is a compact subspace $K$ of $X$ such
that $\mathcal{F}(\tilde{C})(U)$ is supported in $K$ i.e.
$\coend{\sigma}\SC{\sigma}{U\cap K}$ depends only on $K$. Since
$C$ is an $X$-based object there are only finitely many $\sigma$
such that $C(\sigma)\neq 0$. These simplices span a subcomplex $K$
and $\coend{\sigma}\SC{\sigma}{U\cap K}$ is supported in $K$.
\end{proof}
\begin{proposition}
 For $C\in\cat{B}(\mathbb{Z},X), \mathcal{F}(C)$ lies in $\cattwo{D}_X$. 
\end{proposition}
\begin{proof}
 By the preceding lemmata $\mathcal{F}(C)$ lies in $\cattwo{C}_X$. It remains to show that it is in fact contained in the full subcategory of Def. \ref{Definition:well generated}. For every $\sigma\in X$, $K(\sigma)$ is in $\cattwo{D}_X$ and hence also the pushforward $\iota_{\sigma\ast}K(\sigma)$. We can view $C[\sigma]$ as an element in $\cattwo{D}_{pt}$. It follows from \cite[6.5]{RaWei10} that $C[\sigma]\otimes \iota_{\sigma\ast}K(\sigma)$ is in $\cattwo{D}_X$. Since the coend $\coend{\sigma}C[\sigma]\otimes \iota_{\sigma\ast}K(\sigma)$ is given by a direct sum modulo the image of a direct sum it is also in $\cattwo{D}_X$.
\end{proof}
\begin{lemma}\label{F commutes with maps induced by a simplicial map}
 For a simplicial map $f:X\rightarrow X'$  and a $(\mathbb{Z},X)$-module chain complex $C$ there is a natural transformation of functors $\eta:f_\ast\mathcal{F}(C)\rightarrow \mathcal{F}(f_\ast C)$ with $\eta_U$ being a homotopy equivalence for every open set $U\subset X'$. Furthermore for $f$ injective, $\eta_U$ is an isomorphism.

\end{lemma}
\begin{proof}
 Let $C$ be in $\cat{B}(\mathbb{Z},X)$. We want to show that the obvious natural map 
\[
\coend{\sigma\in X}C[\sigma]\otimes (f\iota_\sigma)_\ast K(\sigma)(U)\rightarrow\coend{\sigma'\in X'}f_\ast C[\sigma']\otimes \iota_{\sigma'\ast}K(\sigma')(U)
\]
is a chain homotopy equivalence for every open $U$ in $X'$. It is sufficient to show this statement for $(\mathbb{Z},X)$-module chain complexes which are  concentrated in one degree in which they are free on one generator. Let $C$ be one of these, i.e. \linebreak
\[
C_i=\left\{
\begin{array}{ll}
 M_\sigma&i=k\\
0 & i\neq k\\
\end{array}
\right.
\]
where for a simplex $\sigma$ in $X$,  $M_{\sigma}$ is free on one generator (cf. Rm. \ref{Generation by free (Z,X)-modules}). We use Rm. \ref{coend as big sum} to rewrite the above in degree $n$ as
\[
M_\sigma(\sigma)_k\otimes S_{n-k}(\sigma\cap f^{-1}(U))\rightarrow M_\sigma(\sigma)_k\otimes S_{n-k}(f(\sigma)\cap U).
\]
Now $\sigma\cap f^{-1}(U)$ is nonempty if and only if $f(\sigma)\cap U$ is and both terms are naturally equivalent to $M_{\sigma}(\sigma)=\mathbb{Z}$. Otherwise both are zero. If $f$ is injective 
\[
S_j(\sigma\cap f^{-1}(U))\xrightarrow{f_\ast} S_j(f(\sigma)\cap U)
\] 
is an isomorphism.
\end{proof}

\section{Map between $\boxtimes$ products}\label{sec:6}
We will need the following lemma.
\begin{lemma} There is homotopy
equivalence
\[
M\boxtimesRA N\simeq\coend{\sigma}M[\sigma]\otimes
N[\sigma]\otimes\triangle_\ast(\sigma)
\]
which is natural in both components.
\end{lemma}
\begin{proof}
 With the natural homotopy equivalence 
\[
 M[\sigma]\otimes N[\sigma]\otimes \triangle_\ast(\sigma)\simeq M[\sigma]\otimes N[\sigma]
\] 
the homotopy coend becomes a homotopy colimit. Since  $\sigma\mapsto M[\sigma]$ is a Reedy cofibrant functor to chain complexes, the colimit computes the homotopy colimit:
\[
M\boxtimesRA N=\colim[\sigma\in X]M[\sigma]\otimes N[\sigma]\simeq\hocolim[\sigma\in X]M[\sigma]\otimes N[\sigma].
\] 
\end{proof}

This is the model  we will be working with. Now we can formulate a local
criterion for nondegeneracy in $\mathbb{B}(\mathbb{Z},X)$.
\begin{lemma}
Given two complexes $C,D\in\mathbb{B}(\mathbb{Z},X)$. For each
$\sigma\in X$ there is a map
\[
L:C\boxtimesRA D\rightarrow C(\sigma)\otimes
D[\sigma]\otimes\triangle_\ast(\sigma,\partial\sigma).
\]
A cycle $\phi\in C\boxtimes D$ is nondegenerate if and only if its
image in 
\[
\prod_\sigma C(\sigma)\otimes
D[\sigma]\otimes\nolinebreak\triangle_\ast(\sigma,\partial\sigma)
\]
is nondegenerate.
\end{lemma}
\begin{proof}
To give a map from the homotopy coend it is sufficient to give a map
from each component
\[
f_\alpha: C[\alpha]\otimes
D[\alpha]\otimes\triangle_\ast(\alpha)\rightarrow C(\sigma)\otimes
D[\sigma]\otimes\triangle_\ast(\sigma,\partial\sigma)
\]
consistent with inclusions
\[
C[\alpha]\otimes D[\alpha]\otimes\triangle_\ast(\beta)\rightarrow
C[\alpha]\otimes D[\alpha]\otimes\triangle_\ast(\alpha)
\]
for $\beta\leq\alpha$ and
\[
C[\gamma]\otimes D[\gamma]\otimes\triangle_\ast(\alpha)\rightarrow
C[\alpha]\otimes D[\alpha]\otimes\triangle_\ast(\alpha)
\]
for $\gamma\geq\alpha$. Define $f_\alpha$ to be the obvious quotient
map if $\alpha=\sigma$ and $0$ otherwise. It is easy to see that the
$f_\alpha$ are consistent in the above sense because of the special
form of the domain: everything which comes from a bigger or a
smaller simplex is quotiented out in the domain. The second statement
follows from \cite[Prop. 2.7]{RaWei90}
\end{proof}
\hfill
\begin{proposition}\label{Map of boxtimes-products}
\hfill
\begin{enumerate}
 \item[a)] 
For $C,D\in\mathbb{B}(\mathbb{Z},X)$ there is a map
\[
H=H_{C,D}:C\boxtimes D\rightarrow \mathcal{F}(C)\boxtimes
\mathcal{F}(D)
\]
natural in both arguments.
\item[b)] Let $\phi$ be a nondegenerate cycle in $C\boxtimes D$. Its image
$H(\phi)$ is also nondegenerate in $\mathcal{F}(C)\boxtimes
\mathcal{F}(D)$.
\end{enumerate}
\end{proposition}
\begin{proof}
a) Given a $C\in\mathbb{B}(\mathbb{Z},X)$, $\mathcal{F}(C)$ can be
expressed as a global coend
\[
\mathcal{F}(C)=\coend{\sigma}C[\sigma]\otimes\iota_{\sigma\ast}
K(\sigma)
\]
as in Rm. \ref{F as global coend}. Since $\sigma$ is a manifold with boundary, there exists a chain
$z_\sigma\in\triangle_{|\sigma|}(\sigma)$ mapping to a fundamental cycle in $\triangle_{|\sigma|}(\sigma,\partial\sigma)$. For two objects
$C,D\in\mathbb{A}(\mathbb{Z},X)$ let
\[
h_1:C[\sigma]\otimes D[\sigma]\otimes\triangle_\ast(\sigma)\rightarrow C[\sigma]\otimes D[\sigma]\otimes \iota_{\sigma_\ast}K(\sigma)\boxtimes \iota_{\sigma_\ast}K(\sigma)
\]
be the composition of $id\otimes\nabla$ and the pushforward on the boxtimes component, where $\nabla$ is the refinement of the Alexander-Whitney diagonal approximation 
\[\nabla: S_\ast(\sigma)\rightarrow
K(\sigma)\boxtimes K(\sigma)
\]
mentioned in Def. \ref{L-fundamental class for cosheaves} composed with the map 
\[ 
 \triangle_\ast(\sigma)\rightarrow S_\ast(\sigma).
\] 
Let
\[
\begin{array}{cl}
C[\sigma]\otimes D[\sigma]\;\otimes& \iota_{\sigma_\ast}K(\sigma)\boxtimes \iota_{\sigma_\ast}K(\sigma)\xrightarrow{\;h_2(U,K_1,K_2)\;}
\vspace{1mm}\\
&C[\sigma]\otimes S_\ast((U,U\setminus K_1)\cap\sigma)\otimes D[\sigma]\otimes S_\ast((U,U\setminus K_2)\cap\sigma)\\
\end{array}
\]
be the composition of the corresponding homotopy projection and the
transposition of the inner components. By the universal property we
get a map to the homotopy limit
\[
\begin{array}{cl}
C[\sigma]\otimes D[\sigma]&\otimes\;\iota_{\sigma_\ast}K(\sigma)\boxtimes \iota_{\sigma_\ast}K(\sigma)\xrightarrow{\;\;h_2\;\;}
\vspace{1mm}\\
&\holim[(U,K_1,K_2)]C[\sigma]\otimes S_\ast((U,U\setminus
K_1)\cap\sigma)\otimes D[\sigma]\otimes S_\ast((U,U\setminus
K_2)\cap\sigma)\vspace{1mm}\\
 &\qquad\qquad\qquad\qquad\qquad\qquad=C[\sigma]\otimes\iota_{\sigma_\ast}K(\sigma)\boxtimes D[\sigma]\otimes
 \iota_{\sigma_\ast}K(\sigma)

\end{array}
\]
By taking the homotopy coend of $h_2\circ h_1$ we get
\[
\coend{\sigma}C[\sigma]\otimes
D[\sigma]\otimes\triangle_\ast(\sigma)\rightarrow\coend{\sigma}C[\sigma]\otimes\iota_{\sigma_\ast}K(\sigma)\boxtimes
D[\sigma]\otimes
 \iota_{\sigma_\ast}K(\sigma)
\]
which we compose with the obvious inclusion
\[
\coend{\sigma}C[\sigma]\otimes\iota_{\sigma_\ast}K(\sigma)\boxtimes
D[\sigma]\otimes
 \iota_{\sigma_\ast}K(\sigma)\rightarrow\coend{\sigma}C[\sigma]\otimes\iota_{\sigma_\ast}K(\sigma)\boxtimes\coend{\sigma}
D[\sigma]\otimes
 \iota_{\sigma_\ast}K(\sigma)
\]
and the latter is $\mathcal{F}(C)\boxtimes \mathcal{F}(D)$ per
construction. The required map is 
\[ 
H=incl\circ\coend{\sigma}(h_2\circ h_1).
\]
b) We make use of the local criterion in \cite[Prop.5.8]{RaWei10}.
For every open set $U$ in $X$ and every $j\geq 0$, we have to show
that the slant product with the corresponding projection of
$H(\phi)$
\[
\setminus H(\phi)_U:\colim[K\subset U]\mathcal{F}(C)(U,U\setminus K)^{n-j}\rightarrow
\mathcal{F}(D)(U)_j
\]
is a chain homotopy equivalence. Assume we have already shown it for
open sets which are unions of open stars in an arbitrary subdivision
of $X$. Then we can use the same argument as in the standard proof of
Poincar\'e duality. Cover an arbitrary open set by unions of open
stars and use the Mayer-Vietoris sequence and Zorn's lemma if
needed. Hence the proof of b) reduces to the next lemma.
\end{proof}

\begin{lemma}
If $U$ is a union of open stars in any (barycentric) finite subdivision $X^{(n)}$
of $X$ then the slant product with the corresponding projection of
$H(\phi)$
\[
\colim[K\subset U]\mathcal{F}(C)(U,U\setminus K)^{n-\ast}\rightarrow
\mathcal{F}(D)(U)_\ast
\]
is a natural chain homotopy equivalence
\end{lemma}
\begin{proof}
We prove the statement for unions of open stars
in the original simplicial complex $X$ first. Let $U$ be an open star
$\opst{\tau}=\bigcup_{\sigma\geq\tau}\mathring{\sigma}$. The system
$(U\setminus K)_K$ with $(U\setminus K)$ homotopy equivalent to
$(U\setminus \hat{\tau})$, where $\hat{\tau}$ is the barycenter of
$\tau$, is cofinal in the system of all $(U\setminus K)$. Therefore
the map in the statement of the lemma can be rephrased
\[
\mathcal{F}(C)(\opst{\tau},\opst{\tau}\setminus
\hat{\tau})^{n-\ast}\rightarrow \mathcal{F}(D)(\opst{\tau})_\ast.
\]
This is a homotopy equivalence if and only if the component of $H(\phi)$ in
\[
\mathcal{F}(C)(\opst{\tau},\opst{\tau}\setminus \hat{\tau})\otimes
\mathcal{F}(D)(\opst{\tau})
\] 
is nondegenerate. We have a commutative diagram

\[
\begindc{0}[35] 
\obj(1,3)[1]{$C\boxtimes D$}
\obj(7,3)[2]{$\mathcal{F}(C)\boxtimes\mathcal{F}(D)$}
\obj(1,1)[3]{$C(\tau)\otimes D[\tau]\otimes\triangle_\ast(\tau,\partial\tau)$} 
\obj(7,1)[4]{$\mathcal{F}(C)(\opst{\tau},\opst{\tau}\setminus \hat{\tau})\otimes
\mathcal{F}(D)(\opst{\tau})$}

\mor{1}{2}{$H$}[1,0] \mor{2}{4}{$proj$}[-1,0] \mor{1}{3}{$
L$}[-1,0]\mor{3}{4}{$ id\otimes\nabla $}[-1,0]
\enddc
\]
Let us take a closer look at the first tensor
factor in the right bottom corner. The coend varies over terms of the form
\[
C[\rho]\otimes
S(\mathring{\rho}\cup\bigcup_{\rho>\sigma\geq\tau}\mathring{\sigma},
\mathring{\rho}\cup\bigcup_{\rho>\sigma\geq\tau}\mathring{\sigma}\setminus
\hat{\tau}).
\]
Using Rm. \ref{coend as big sum} we can view it as a sum
\[
C(\tau)\otimes S(\mathring{\tau},\mathring{\tau}\setminus\hat{\tau})\oplus \bigoplus_{\rho>\tau}C(\rho)\otimes
S(\mathring{\rho}\cup\bigcup_{\rho>\sigma\geq\tau}\mathring{\sigma},
\mathring{\rho}\cup\bigcup_{\rho>\sigma\geq\tau}\mathring{\sigma}\setminus
\hat{\tau}).
\]
By excision we have
\[
S(\mathring{\rho}\cup\bigcup_{\rho>\sigma\geq\tau}\mathring{\sigma},\mathring{\rho}\cup\bigcup_{\rho>\sigma\geq\tau}\mathring{\sigma}\setminus\hat{\tau})\simeq S(\mathring{\rho},\mathring{\rho})\simeq 0
\]
and deduce that $\mathcal{F}(C)(\opst{\tau},\opst{\tau}\setminus\hat{\tau})\simeq C(\tau)\otimes S(\mathring{\tau},\mathring{\tau}\setminus\hat{\tau})\simeq S^{|\tau|}C(\tau)$.

Because of Lemma \ref{new is old} the second tensor factor $\mathcal{F}(D)(\opst{\tau})$ is naturally
equivalent to $D[\tau]$. Thus we have shown the natural equivalence 
\[
 \mathcal{F}(C)(\opst{\tau},\opst{\tau}\setminus\hat{\tau})\otimes\mathcal{F}(C)(\opst{\tau})\simeq S^{|\tau|} C(\tau)\otimes D[\tau]
\]
and can rewrite the above commutative square as
\[
\begindc{0}[35] 
\obj(1,3)[1]{$C\boxtimes D$}
\obj(7,3)[2]{$\mathcal{F}(C)\boxtimes\mathcal{F}(D)$}
\obj(1,1)[3]{$C(\tau)\otimes D[\tau]\otimes\triangle_\ast(\tau,\partial\tau)$} 
\obj(7,1)[4]{$S^{|\tau|}C(\tau)\otimes D[\tau]$}

\mor{1}{2}{$H$}[1,0] \mor{2}{4}{$proj$}[-1,0] \mor{1}{3}{$
L$}[-1,0]\mor{3}{4}{$ id\otimes\nabla $}[-1,0]
\enddc
\]

The projection of $H(\phi)$ is nondegenerate if the anticlockwise composition maps $\phi$ to a nondegenerate cycle. Since the map
\[
C(\tau)\otimes
D[\tau]\otimes\triangle_\ast(\tau,\partial\tau)=S^{|\tau|}C(\tau)\otimes D[\tau]\xrightarrow{id\otimes\nabla} S^{|\tau|}C(\tau)\otimes D[\tau]
\]
is homotopic to identity this is obvious.\\\\
Let $U$ now be a finite union of open stars. Since the intersection
of two open stars is an open star the statement of the lemma for $U$
follows from a Mayer-Viertoris argument and the corresponding case
of a single star.
To deal with (unions of) open stars in a finite subdivision $X^{(n)}$ we proceed entirely analogously to the above by considering the images $C^{(n)}, D^{(n)}$ in the category $\cattwo{K}(\mathbb{Z},X^{(n)})$ of $\mathbb{Z}$-modules over the $n$-th barycentric subdivision. That this makes sense and, more importantly, that the nondegeneracy of a cycle in $C\boxtimes D$ is preserved after passing to subdivisions (and looking at duality properties over smaller open stars) is the content of the next lemma. 

\end{proof}
 
\begin{lemma}
 Let $X'$ denote the barycentric subdivision of the simplicial complex $X$. There exists a functor of weak algebraic bordism categories from $\cat{B}(\mathbb{Z},X)$ to $\cat{B}(\mathbb{Z},X')$. 
\end{lemma}
\begin{proof}
To some extent this seems to be folklore. Let $C$ be an object in $\cat{B}(\mathbb{Z},X)$. Following the same idea as in extending cosheaves over open stars to cosheaves over arbitrary open sets we can set
\[
C'[\hat{\sigma}_{i_1}\ldots\hat{\sigma}_{i_k}]=C[\sigma_{i_k}]\otimes \triangle_\ast(\hat{\sigma}_{i_1}\ldots\hat{\sigma}_{i_k}).
\]
This value of $C'$ over $\hat{\sigma}_{i_1}\ldots\hat{\sigma}_{i_k}$ is now determined by
\[
C'(\hat{\sigma}_{i_1}\ldots\hat{\sigma}_{i_k})=C'[\hat{\sigma}_{i_1}\ldots\hat{\sigma}_{i_k}]\,\, /\colim[\sigma'\geq\hat{\sigma}_{i_1}\ldots\hat{\sigma}_{i_k}] C'[\sigma'].
\]
Working out the effect on the morphisms one can show that $C'$ is in $\cat{B}(\mathbb{Z},X')$. This assignment is also similar in spirit to the algebraic subdivision functor of Adams-Florou (cf. \cite{AF12}). Using his explicit description of duals one should be able to show that subdivision is a functor of algebraic bordism categories. 
\end{proof}
\section{Equivalence of $L$-spectra} \label{sec:7}
The next theorem presents our main result.
\begin{theorem}\label{Main Theorem}
The functor $\mathcal{F}$ of Def. \ref{THE functor} induces equivalences of spectra
\[
\SymL(\cattwo{K}_{(\mathbb{Z},X)})\xrightarrow{\simeq} \SymL(\cattwo{K}_X), \qquad \QuadL(\cattwo{K}_{(\mathbb{Z},X)})\xrightarrow{\simeq} \QuadL(\cattwo{K}_X)
\]
and in particular for every $n\geq 0$ isomorphisms
\[
L^n(\cattwo{K}_{(\mathbb{Z},X)})\xrightarrow{\cong} L^n(\cattwo{K}_X), \textnormal{  and  } L_n(\cattwo{K}_{(\mathbb{Z},X)})\xrightarrow{\cong} L_n(\mathcal{K}_X).
\]
\end{theorem}
\begin{proof}
We treat the symmetric case only, the quadratic one being completely analogous. Observe that $\mathcal{F}$ maps $\cat{C}(\mathbb{Z},X)$-contractible objects of $\cat{B}(\mathbb{Z},X)$ to $\cattwo{C}_X$-con-tractible cosheaves in $\cattwo{D}_X$. The natural transformation $H$ of Prop. \ref{Map of boxtimes-products} makes the functor $\mathcal{F}$ into a functor of weak algebraic bordism categories. Thus, by Prop. \ref{An equivariant functor induces maps of spectra} there is an induced map of spectra
\[
\SymL(\cattwo{K}_{(\mathbb{Z},X)})\rightarrow \SymL(\mathcal{K}_X).
\] 
For a map of simplicial complexes $f:X\rightarrow Y$ we get a square
\[
\begindc{0}[25]
\obj(1,3)[1]{$\SymL(\cattwo{K}_{(\mathbb{Z},X)})$}
\obj(4,3)[2]{$\SymL(\cattwo{K}_X)$}
\obj(1,1)[3]{$\SymL(\cattwo{K}_{(\mathbb{Z},Y)})$}
\obj(4,1)[4]{$\SymL(\mathcal{K}_Y)$}

\mor{1}{2}{$\mathcal{F}$}[1,0] 
\mor{2}{4}{$f_\ast $}[-1,0] 
\mor{1}{3}{$f_\ast $}[-1,0]
\mor{3}{4}{$\mathcal{F} $}[1,0]
\enddc
\]
which is commutative because of Lemma \ref{F commutes with maps induced by a simplicial map}. Thus $\mathcal{F}$ is a natural transformation between functors $X\mapsto\SymL(\cattwo{K}_{(\mathbb{Z},X)}$ and $X\mapsto\SymL(\cattwo{K}_X)$ and both are homotopy invariant and excisive because of Theorems \ref{X to (Z,X)-modules is homology theory} and \ref{X to cosheaves is homology theory}. For a point $pt$ there are isomorphisms
\[
\SymL(\cattwo{K}_{\mathbb{Z},pt)})=\SymL(\Lambda(\mathbb{Z}))\cong\SymL(\mathbb{Z})\cong\SymL(\cattwo{K}_{pt})
\]
which implies that $\mathcal{F}$ is an isomorphism of homology theories.

\end{proof}
\section{Final remarks}\label{sec:8}
 With insignificantly more effort all the results above can be proved for the simplicial complex $X$ being replaced by a $\triangle$-set: the $L$-homology description of chain complexes of $\mathbb{Z}$-modules parametrised by a $\triangle$-set $X$ (cf. \cite{RaWei12}) can be canonically and naturally identified with the $L$-theory of $\cattwo{K}_{|X|}$.

Furthermore one can generalize the main theorem from $\mathbb{Z}$-coefficients to coefficients in any commutative ring $R$ with the trivial involution. Replacing the category of $\mathcal{O}(X)$-based free abelian groups $\cattwo{A}_X$ in Def. \ref{definition O(X)-modules} by the analogous category of  $\mathcal{O}(X)$-based free $R-$modules $\cattwo{A}_X\otimes_\mathbb{Z}R$ one constructs the weak algebraic bordism category $\cattwo{K}^R_X$, such that $\cattwo{K}^\mathbb{Z}_X=\cattwo{K}_X$. A generalisation of Thm. \ref{Main Theorem} provides then a functor of weak algebraic bordism categories
\[
\cattwo{K}_{(R,X)}\rightarrow \cattwo{K}^R_X
\]
inducing an isomorphism on symmetric and quadratic $L$-groups.\\

As pointed out in the introduction, a description of $L$-homology is the first step in the description of the assembly map and the explicit construction of the total surgery obstruction. By constructing the functor $F$ and showing the main theorem, this work merely builds a rope bridge between the combinatorial framework of Ranicki and the more flexible but less explicit language of Ranicki-Weiss.\linebreak

We want to end with a very brief reminder of $L$-theory descriptions which exist in the literature.
An honest sheaftheoretic description of $L$-homology, assembly map and total surgery obstruction was undertaken by Hutt in \cite{Hutt}. Unfortunately there is a mistake in this preprint and it was never published. In \cite{Woo08} Woolf considers a triangulated version of Hutt's framework. Under the assumption that $R$ is a regular Noetherian ring of finite Krull dimension and $\frac{1}{2}\in R$ he identifies Ranicki's construction of (free) symmetric $L$-homology $L^\ast(\cattwo{K}_{(R,X)})$ with the Witt groups (in the sense of Balmer, see \cite{Bal05}) of the triangulated category ( with duality ) of constructible (w.r.t. the stratification induced by the simplicial structure) sheaves of $R$-module complexes $W^c_\ast(X)$. Putting these functors, along with $\pi_\ast \SymL(\cattwo{K}_X)$, in one diagram we get

\[
\begindc{0}[5]
\obj(10,20)[1]{$W^c_\ast(X)$} 
\obj(30,20)[2]{$\pi_\ast\SymL(\mathcal{K}_X)$}
\obj(20,10)[3]{$H_\ast(X;\SymL(R))$}
\mor{2}{1}{} [\atright,\dashline]
\mor{3}{2}{} 
\mor{3}{1}{$ $}[\atleft,\solidarrow]
\enddc
\]
where all three terms are isomorphic via the solid arrows. It seems natural to search for a construction of a canonical, geometric morphism for the dashed line.\\

On the other hand in his doctoral thesis \cite{Epp07} Eppelmann gives a geometric description of a 2-connected cover of symmetric $L$-homology as a (singular) bordism homology $\Omega_\ast^{IP}(X)$ of spaces satisfying integral Poincar\'e duality in intersection homology. A natural question is how this description fits into the above diagram.\\ 

In comparing our work to \cite{Woo08}, it is important to remark that in the setup of \cite{Woo08} the condition $\frac{1}{2}\in R$ cannot be dropped and Woolf's result cannot be generalised to integral coefficients. This  restriction is specific to Balmer's theory. The $L$-theory of Ranicki-Weiss cosheaf complexes is different in nature since the duality is only given on the homotopy category. The necessity of inverting 2 however, also finds its way into \cite{RaWei10}. To make their proof of topological invariance of rational Pontryagin classes independent of difficult arguments of \cite{KS77}, Ranicki and Weiss introduce the idempotent completion $r\mathcal{D}_X$ of the category $\mathcal{D}_X$ (see Def. \ref{Definition:well generated}) underlying $\cattwo{K}_X$. Excision for the functor  $X\mapsto\SymL(r\mathcal{D}_X)$ is only proved up to 2-torsion (cf. \cite[Thm. 8.3]{RaWei10}). 

\section{Appendix}\label{sec:9}
It is apparent from section 3 that we are using the notion of homotopy limits (and colimits) of chain complexes in this article. The classical source quoted at this point is \cite{BK72}. This deals with diagrams of simplicial sets or diagrams of topological spaces, which amounts to the same. Since the category of positively graded chain complexes of $R$-modules is Quillen equivalent to the category of simplicial $R$-modules, we can basically use the original constructions for simplicial sets to get the right notions for chain complexes. A framework for homotopy limits in general model categories can be found in \cite{Hi03} or \cite{DHKS04}. For our purposes only the properties of homotopy (co)limits matter. Nevertheless, it might be convenient for the reader to see some explicit models. This section is highly non-original and the author benefited significantly from the exposition in \cite{Du14}.
\begin{definition}
 Let $C$ be a simplicial chain complex of abelian groups i.e. a functor from $\triangle^{op}$ to $\ChAb$. Denote now by $D^\ast$ the cosimplicial object in $\ChAb$ given by taking the simplicial chain complex of the standard simplex $\triangle^n$ viewed as a simplicial space. Define the \textit{realisation} of $C$  by the coequaliser
\[
\{C\}=\coeq \big[ \bigoplus_{[n]\rightarrow[k]} C^k\otimes D^n \rightrightarrows \bigoplus_{[n]} C^n\otimes D^n \big]
\]
where the top map is induced by $[n]\rightarrow [k]$ and the bottom by the standard map $\triangle^k\rightarrow \triangle^n$.\\
Dually, let $C$ be a cosimplicial object in $\ChAb$.  Define its \textit{totalisation} by the equaliser
\[
\Tot(C)=\eq\big[ \prod_{[n]} \Hom(D^n, C^n) \rightrightarrows \prod_{[n]\rightarrow [k]} \Hom(D^n, C^k) \big].
\]
\end{definition}

\begin{definition}
 Let $\mathcal{C}$ be a (small) category and $F:\mathcal{C}\rightarrow \ChAb$ a functor to chain complexes of abelian groups. Define its \textit{simplicial} replacement $\srep(F)$ by the simplicial object in $\ChAb$ given in degree $n$ by 
\[
\srep(F)_n=\bigoplus_{i_n\rightarrow i_{n-1}\rightarrow\cdots\rightarrow i_0} F(i_n)
\]
with $i_n\rightarrow i_{n-1}\rightarrow\cdots\rightarrow i_0$ being a chain of composable maps in $\mathcal{C}$. For $0\leq j\leq n$, the degeneracy map $s_j:\srep(F)_n\rightarrow\srep(F)_{n+1}$ sends $F(i_n)$ in the component with index
\[
i_n\rightarrow i_{n-1}\rightarrow\cdots\rightarrow i_0 
\]
to $F(i_n)$ at
\[
i_n\rightarrow i_{n-1}\rightarrow \cdots\rightarrow i_j\xrightarrow{id} i_j\rightarrow\cdots\rightarrow i_0.
\]
For $0\leq j < n$ the face map $d_j:\srep(F)_n\rightarrow\srep(F)_{n-1}$ sends the copy of $F(i_n)$ at
\[
i_n\rightarrow i_{n-1}\rightarrow\cdots\rightarrow i_0
\]
to $F(i_n)$ sitting at
\[
i_n\rightarrow i_{n-1}\rightarrow i_{j+1}\rightarrow i_{j-1}\rightarrow\cdots\rightarrow i_0
\]
with $i_{j+1}\rightarrow i_{j-1}$ being the composition $i_{j+1}\rightarrow i_{j}\rightarrow i_{j-1}$. For $j=n$ the face map $d_n$ maps $F(i_n)$ at
\[
i_n\rightarrow i_{n-1}\rightarrow\cdots\rightarrow i_0
\]
to $F(i_{n-1})$ at 
\[
i_n-1\rightarrow i_{n-1}\rightarrow\cdots\rightarrow i_0
\]
via $F(i_n\rightarrow i_{n-1})$.\\

Dually, define the cosimplicial replacement $\csrep(F)$ as a cosimplicial object in the category $\ChAb$ given in degree $n$ by\\
\[
\csrep(F)_n=\prod_{i_0\rightarrow\cdots\rightarrow i_n} F(i_n)
\]
with analogous coface and codegeneracy maps.
\end{definition}
\begin{definition}
 Let $F$ be a functor from a small category $\mathcal{C}$ to chain complexes of abelian groups $\ChAb$. Define the homotopy limit of $F$ as
\[
\holim[\mathcal{C}] F=\Tot(\csrep(F)).
\]
and dually the homotopy colimit of $F$ as
\[
\hocolim[\mathcal{C}] F=\{\srep(F)\}.
\]
\end{definition}
For those familiar with homotopy limits and model categories, the next two propositions are fairly basic. We refer e.g. to \cite{Hi03} for a thorough discussion.
\begin{proposition}
Let $F_1,F_2$ be functors from a small category $\mathcal{C}$ to bounded chain complexes of projective $R$-modules, for any ring $R$. Assume there is a natural transformation $\eta:F_1\Rightarrow F_2$ such that for every $c\in\mathcal{C}$, $\eta_c$ is a weak equivalence. Then the induced map
\[
\hocolim F_1\rightarrow \hocolim F_2
\]
is also a weak equivalence.
\end{proposition}
A poset $K$ is a Reedy category. A functor $F$ from $K$ to chain complexes is called \textit{Reedy cofibrant} if and only if for every $b\in K$ the induced map
\[
\colimtwo{a\in K}{a\neq b,a\rightarrow b}F(a)\rightarrow F(b)
\]
is a cofibration i.e. a degreewise split injection. We will make use of the following
\begin{proposition}
Let $X$ be a simplicial complex viewed as a poset in the obvious way. If a functor $F:X\rightarrow \ChAb$ is Reedy cofibrant the colimit computes the homotopy colimit i.e. the canonical map
\[
\colim F\xrightarrow{\simeq} \hocolim F
\] 
is a weak equivalence.
\end{proposition}
The dual statements for homotopy limits are also valid.\\

What we also made use of in the main body of the article is a homotopy version of
a coend. Following \cite[IX,6]{Mac98} the coend of a given functor
\[
F:\mathcal{C}^{op}\times\mathcal{C}\rightarrow \mathcal{D}
\]
might be described as a colimit over the twisted arrow category $\mathcal{C}^\natural$ in the following way. Let the objects of $\mathcal{C}^\natural$ be morphisms $f:a\rightarrow b$ of $\mathcal{C}$. The morphisms between $f$ and $g$ are pairs of morphisms $(h,j)$ in $\mathcal{C}$ satisfying $jgh=f$. There is a target-source functor $ts$ from $\mathcal{C}^\natural$ to $\mathcal{C}^{op}\times\mathcal{C}$. The bifunctor $F$ gives rise to a functor $F^\natural:\mathcal{C}^\natural\rightarrow\mathcal{D}$
via
\[
\begindc{0}[6]
\obj(10,20)[1]{$\mathcal{C}^\natural$}
\obj(10,10)[2]{$\mathcal{C}^{op}\times\mathcal{C}$}
\obj(20,10)[3]{$\mathcal{D}$}
\mor{1}{2}{$ts$}[-1,0]
\mor{2}{3}{$F$}[-1,0]
\mor{1}{3}{$F^\natural$}[1,0]
\enddc
\]

The coend of $F$ can be defined via
\[
\coend{a\in\mathcal{C}} F(a,a) = \colim[(a\rightarrow b) \in \mathcal{C}^\natural] F^\natural(a\rightarrow b)=\colim[(a\rightarrow b) \in \mathcal{C}^\natural] F(b,a).
\] 
In the same spirit define the homotopy coend.
\begin{definition}
 Let $\mathcal{C}$ be a small category and $F$ a functor from $\mathcal{C}^{op}\times\mathcal{C}$ to chain complexes of abelian groups $\ChAb$. Its \textit{homotopy coend} is defined by
\[
\hocoend{a\in\mathcal{C}} F=\hocoend{} F = \hocolim F^\natural=\hocolim[(a\rightarrow b)\in\mathcal{C}^\natural]F(b,a).
\]
\end{definition}

\bibliographystyle{amsalpha}
\bibliography{Biblio}

\end{document}